\documentclass[12pt]{article}
\usepackage[utf8]{inputenc}
\usepackage{amsthm}
\usepackage{amsfonts}
\usepackage{amsmath}
\usepackage{amssymb}
\usepackage{colonequals}
\usepackage{url}
\usepackage{hyperref}
\usepackage{graphicx}
\newtheorem{theorem}{Theorem}
\newtheorem{corollary}[theorem]{Corollary}

\newtheorem{lemma}[theorem]{Lemma}

\newtheorem{observation}[theorem]{Observation}

\newcommand{\GG}{\mathcal{G}}
\newcommand{\ch}{\mathrm{ch}}

\begin{document}
\title{Clustered coloring of $(\text{path}+2K_1)$-free graphs on surfaces}
\author{Zden\v{e}k Dvo\v{r}\'ak
	\thanks{Charles University, Prague, Czech Republic. E-mail: \protect\href{mailto:rakdver@iuuk.mff.cuni.cz}{\protect\nolinkurl{rakdver@iuuk.mff.cuni.cz}}.
	Supported by project 22-17398S (Flows and cycles in graphs on surfaces) of Czech Science Foundation.}
}
\date{}

\maketitle

\begin{abstract}
Esperet and Joret proved that planar graphs with bounded maximum degree are
3-colorable with bounded clustering.  Liu and Wood asked whether the conclusion
holds with the assumption of the bounded maximum degree replaced by assuming that
no two vertices have many common neighbors.  We answer this question in positive,
in the following stronger form: Let $P''_t$ be the complete join of two isolated vertices
with a path on $t$ vertices.  For any surface $\Sigma$, a subgraph-closed class of graphs
drawn on $\Sigma$ is 3-choosable with bounded clustering if and only if there exists $t$
such that $P''_t$ does not belong to the class.
\end{abstract}

A famous conjecture of Hadwiger postulates that for every positive integer $k$, every $K_{k+1}$-minor-free
graph is properly $k$-colorable.  This conjecture is widely open, only known to be true for $k\le 6$ by a reduction
to the Four Color Theorem~\cite{robertsonseymourthomas} and with the best known upper bound on the
chromatic number of $K_{k+1}$-minor-free graphs in general being superlinear~\cite{delcourt2021reducing}, on the order of $O(k\log\log k)$.

Hence, it is natural to ask whether Hadwiger's conjecture holds at least in some relaxed sense.
One of the best studied relaxations is in the setting of \emph{clustered coloring}.
Throughout the paper, by a \emph{coloring} of a graph $G$ we mean an arbitrary assignment $\varphi$ of colors to vertices
of $G$; in particular, we do not require the colorings to be proper, unless specified otherwise.
A subgraph $H$ of $G$ is \emph{monochromatic} in $\varphi$ if $\varphi(u)=\varphi(v)$ for all $u,v\in V(H)$.
A \emph{cluster} is the vertex set of a maximal monochromatic connected subgraph, and the \emph{clustering} of
a coloring $\varphi$ is the size of the largest cluster.  Note that a proper coloring is exactly a coloring with
clustering one, and on the other extreme, every graph $G$ has a coloring by one color with clustering at most $V(G)$.
Hence, it does not make sense to speak about the clustered chromatic number of a graph without specifying the cluster
size.  However, a definition focusing purely on the number of colors is possible for graph classes.
The \emph{clustered chromatic number} of a graph class $\GG$ is the minimum integer $c$ such that
for some positive integer $\gamma$, every graph in $\GG$ has a coloring by at most $c$ colors with clustering at most $\gamma$.
If no such integer $c$ exists, the clustered chromatic number of $\GG$ is $\infty$.

What is the clustered chromatic number of the class of $K_{k+1}$-minor-free graphs?  Edwards et al.~\cite{edwards2014relative}
give a construction showing that it is at least $k$ (actually, they show a stronger claim: For every $\gamma$, there
exists a $K_{k+1}$-minor-free graph that in every coloring by $k-1$ colors contains a monochromatic subgraph of maximum degree at least $\gamma$).
On the other hand, if Hadwiger's conjecture is true, then the clustered chromatic number of $K_{k+1}$-minor-free graphs clearly is at most $k$.
Building upon a series of improved bounds~\cite{kawarabayashi2007relaxed,wood2010contractibility,edwards2014relative,liusmall,van2018improper},
Dvořák and Norin~\cite{islands} announced a proof of this relaxed version of Hadwiger's conjecture (the full details are still unpublished;
recently, Dujmovi\'{c} et al.~\cite{clushadalt} gave a proof by a different method).

As an important special case, the planar graphs have clustered chromatic number at most four.  This follows from the Four Color Theorem,
but much simpler proofs are known~\cite{cowen1986defective,islands}.  It is also known that this bound cannot be improved,
based on the following standard observation.  For a positive integer $t$, let $P''_t$ denote the complete join of two isolated vertices with a $t$-vertex path
(see Figure~\ref{fig-forb} for an illustration).
\begin{figure}
\begin{center}
\includegraphics{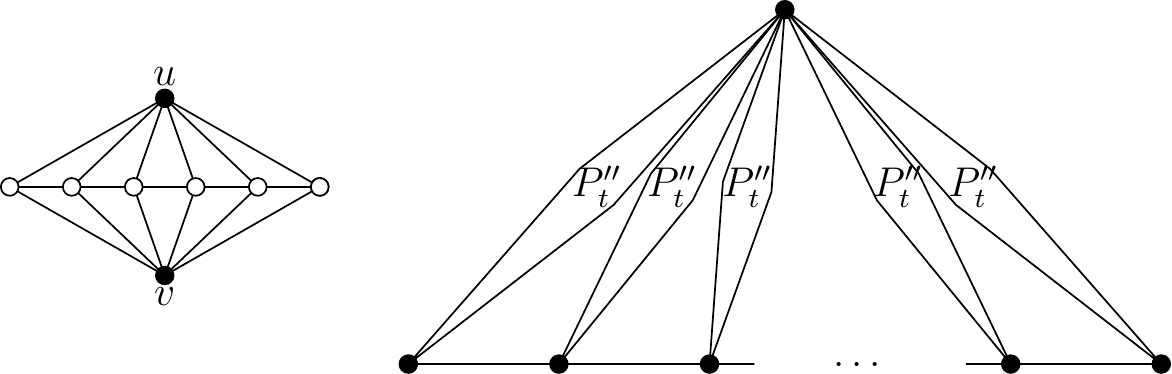}
\end{center}
\caption{The graph $P''_6$ and a construction of a planar graph that is not 3-colorable with small clustering.}\label{fig-forb}
\end{figure}

\begin{observation}\label{obs-crucial}
Let $\gamma$ be a positive integer, let $t=(2\gamma-1)(\gamma+1)$ and let $u$ and $v$ be the two vertices of $P''_t$ of degree $t$.
If $\varphi$ is a $3$-coloring of $P''_t$ of clustering at most $\gamma$, then $\varphi(u)=\varphi(v)$.
\end{observation}
\begin{proof}
Suppose for a contradiction that $\varphi(u)=1$ and $\varphi(v)=2$.  Consider the $t$-vertex path $P=P''_t-\{u,v\}$.
Since $\varphi$ has clustering at most $\gamma$, at most $\gamma-1$ vertices of $P$ have color $1$ and at most $\gamma-1$ vertices
of $P$ have color two.  Hence, the subgraph of $P$ induced by vertices of color $3$ has at most $2\gamma-1$ components,
and since each monochromatic component has size at most $\gamma$, we have $t=|V(P)|\le 2\gamma-2 + (2\gamma-1)\gamma=(2\gamma-1)(\gamma+1)-1$.
This is a contradiction.
\end{proof}
Based on this observation, it is easy to construct a planar graph with no $3$-coloring with clustering at most $\gamma$,
e.g., as illustrated in Figure~\ref{fig-forb}.  However, the number of colors can be improved if we restrict the attention to graphs with bounded maximum degree,
even in more general setting of graphs on surfaces.
\begin{theorem}[Esperet and Joret~\cite{espjor}]\label{thm-maxdeg}
For every integer $\Delta\ge 6$ and any surface $\Sigma$, the class of graphs of maximum degree at most $\Delta$ drawn on $\Sigma$
has clustered chromatic number three.
\end{theorem}
Let us remark that the number of colors cannot be improved, since already the class of planar graphs
of maximum degree six does not have clustered chromatic number two by the Hex Lemma (see e.g. \cite{wood2018defective} for details).
Liu and Wood~\cite{liu2019clustered} substantially strengthened Theorem~\ref{thm-maxdeg}, showing that the
same is true for any class of graphs with bounded layered treewidth (we do not need the precise definition of this concept;
it suffices to know that graphs drawn on any fixed surface have bounded layered treewidth, but there also are many other
graph classes with this property).  Actually, they proved even
stronger result.  Throughout the paper, by an \emph{$F$-free} graph we mean a graph that does not contain a (not necessarily induced) subgraph isomorphic to $F$.
\begin{theorem}[Liu and Wood~\cite{liu2019clustered}]\label{thm-layer}
For any positive integers $c$ and $t$, any class of $K_{c,t}$-free graphs
of bounded layered treewidth has clustered chromatic number at most $c+2$.
\end{theorem}
Note that a graph has maximum degree at most $\Delta$ if and only if it is $K_{1,\Delta+1}$-free,
and thus Theorem~\ref{thm-layer} with $c=1$ implies Theorem~\ref{thm-maxdeg}.
They asked whether this result can be strengthened for planar graphs---specifically, do $K_{2,t}$-free planar graphs
have clustered chromatic number at most three?

Our main result is the positive answer to this question, even for graphs drawn on any fixed surface.
Actually, we show that the claim holds for $P''_t$-free graphs, whose relevance to the
problem is clear from Observation~\ref{obs-crucial}.

\begin{theorem}\label{thm-main1}
For every surface $\Sigma$ and positive integer $t$, the class of $P''_t$-free graphs drawn on $\Sigma$
has clustered chromatic number at most three.
\end{theorem}

The proof of Theorem~\ref{thm-main1} is inspired by the argument used by myself and Norin~\cite{weak} to extend
Theorem~\ref{thm-maxdeg} to the list coloring setting.
For an assignment $L$ of lists to vertices of a graph $G$, an \emph{$L$-coloring}
is any function $\varphi:V(G)\to\bigcup_{v\in V(G)} L(v)$ such that $\varphi(v)\in L(v)$ for every $v\in V(G)$.
For a positive integer $c$, a \emph{$c$-list-assignment} is an assignment of lists of size $c$.
For a positive integer $\gamma$, a graph $G$ is \emph{$c$-choosable with clustering at most $\gamma$}
if $G$ has an $L$-coloring with clustering at most $\gamma$ for every $c$-list-assignment $L$.
Note that this generalizes the standard notion of choosability: A graph is $c$-choosable if and only if
it is $c$-choosable with clustering one.
The \emph{clustered choosability} of a graph class $\GG$ is the minimum number $c$ such that
for some positive integer $\gamma$, every graph in $\GG$ is $c$-choosable with clustering at most $\gamma$.

Thomassen~\cite{thomassen1994} famously proved that planar graphs are 5-choosable, while Voigt~\cite{voigt1993}
gave a construction of non-4-choosable planar graphs.  In contrast, the argument of~\cite{islands}
shows that planar graphs have clustered choosability at most four.  Wood~\cite[Open Problem 18]{wood2018defective}
asked whether Theorem~\ref{thm-maxdeg} generalizes to the list coloring setting; i.e., is it true that for every $\Delta\ge 6$, the class of planar graphs (or more generally,
graphs drawn on a fixed surface) of maximum degree at most $\Delta$ has clustered choosability three?
In~\cite{weak}, we gave a positive answer to this question.  The proof of Theorem~\ref{thm-main1}
is based on the idea developed in~\cite{weak} and naturally gives the result in the list coloring setting as well.

\begin{theorem}\label{thm-main2}
For every surface $\Sigma$ and positive integer $t$, the class of $P''_t$-free graphs drawn on $\Sigma$
has clustered choosability at most three.
\end{theorem}

Let us note a variant of Observation~\ref{obs-crucial} that shows that forbidding $P''_t$ is actually necessary
in this setting.

\begin{observation}\label{obs-crucialchoos}
Let $\gamma$ be a positive integer and let $t=9(2\gamma-1)(\gamma+1)$.  The graph $P''_t$ is not 3-choosable with clustering at most $\gamma$.
\end{observation}
\begin{proof}
Let $u$ and $v$ be the two vertices of degree $t$, and let us subdivide the path of $P''_t$ into nine vertex-disjoint paths $P_{i,j}$ for $i\in\{1,2,3\}$ and $j\in\{4,5,6\}$
of length $(2\gamma-1)(\gamma+1)$.  Let $L(u)=\{1,2,3\}$, $L(v)=\{4,5,6\}$, and for $i\in L(u)$ and $j\in L(v)$, let us set the list of every vertex of $P_{i,j}$ to $\{i,j,7\}$.
For any $L$-coloring $\varphi$, the argument from the proof of Observation~\ref{obs-crucial} shows that $\varphi$ restricted to the subgraph induced by $\{u,v\}\cup V(P_{\varphi(u),\varphi(v)})$
has a cluster of size more than $\gamma$.
\end{proof}

In order to deal with non-contractible triangles in the proof, we need a version of Theorem~\ref{thm-main2}
that allows some of the vertices of the graph to be precolored.
For integers $p$, $c$, and $\gamma$, we say that a graph $G$ is \emph{$p$-precoloring-$c$-choosable with
clustering at most $\gamma$} if for every $c$-list-assignment $L$ for $G$ and every set $X\subseteq V(G)$ of size at most $p$, every $L$-coloring of $X$ extends to
an $L$-coloring of $G$ with clustering at most $\gamma$ (and \emph{$p$-precoloring-$c$-colorable with clustering at most $\gamma$} if this is the
case for the list assignment giving each vertex the same list of size $c$).  The \emph{clustered $p$-precoloring-choosability} (resp. \emph{clustered $p$-precoloring-chromatic number})
of a graph class $\GG$ is the smallest integer $c$ such that for some positive integer $\gamma$,
every graph in $\GG$ is $p$-precoloring-$c$-choosable (resp. $p$-precoloring-$c$-colorable) with clustering at most $\gamma$.
The \emph{clustered precoloring-choosability} of a graph class $\GG$
is the smallest integer $c$ such that for every integer $p$, the class $\GG$ has clustered $p$-precoloring-choosability at most $c$;
in other words, there exists a function $\gamma:\mathbb{N}\to\mathbb{N}$ such that for every integer $p$,
every graph from $\GG$ is $p$-precoloring-$c$-choosable with clustering at most $\gamma(p)$.

\begin{theorem}\label{thm-main}
For every surface $\Sigma$ and positive integer $t$, the class of $P''_t$-free graphs drawn on $\Sigma$
has clustered precoloring-choosability at most three.
\end{theorem}

Let us summarize our results, combining Theorem~\ref{thm-main} and Observations~\ref{obs-crucial} and \ref{obs-crucialchoos}.

\begin{corollary}
For any subgraph-closed class $\GG$ of graphs drawn on a fixed surface, the following
claims are equivalent:
\begin{itemize}
\item $\GG$ has clustered precoloring-choosability at most three.
\item $\GG$ has clustered choosability at most three.
\item $\GG$ has clustered $2$-precoloring-chromatic number at most three.
\item There exists a positive integer $t$ such that $P''_t\not\in \GG$.
\end{itemize}
\end{corollary}

Clustered chromatic number at most three does not have a similar nice characterization in graphs on surfaces---for every $t$,
the graph $P''_t$ is properly 3-colorable for every $t$, and yet using Observation~\ref{obs-crucial} we can combine multiple
copies of $P''_t$ into a planar graph without a 3-coloring with small clustering in many different ways.
Let us also note that clustered $1$-precoloring-chromatic number is equal to clustered chromatic number by symmetry among the colors.

Finally, let us remark that Theorem~\ref{thm-main1} cannot be generalized to the setting of graphs with forbidden minors:
Consider a large triangulated grid together with a universal vertex; this graph only has one vertex of large
degree and does not contain $K_6$ as a minor.  Consider any 3-coloring of this graph.  By symmetry, we can assume that
the universal vertex $u$ has color $3$.  If the cluster of $u$ is small, then there is a large triangulated subgrid on which
only the colors $1$ and $2$ are used, and Hex lemma implies that there is a long monochromatic path in this subgrid.

The rest of the paper is devoted to the proof of Theorem~\ref{thm-main}.
In Section~\ref{sec-islands}, we recall the notion of islands and a key result on their existence in sparse
graphs, noting that graphs on surfaces are just very slightly too dense for this result to imply clustered 3-choosability.
In~\cite{weak}, we introduced the notion of sparsifiers, a variant of reducible configurations designed to deal with this issue.
In Section~\ref{sec-sparsifiers}, we introduce the particular sparsifiers needed to prove Theorem~\ref{thm-main},
and in Section~\ref{sec-disch} we give a discharging argument showing that sparsifier-free graphs are sparse
enough to be amenable to the method from Section~\ref{sec-islands}.  This is enough to prove Theorem~\ref{thm-main}
for graphs without non-facial triangles.  In Sections~\ref{sec-stacks} and \ref{sec-extend}, we give a precoloring extension argument
used to deal with non-facial but contractible triangles (the argument from this section is again closely related
to the one given in~\cite{weak} to deal with the same issue).  The proof is finished by another precoloring extension argument
dealing with the non-contractible triangles in Section~\ref{sec-general}.

Let us remark that although the main ideas of the proof come from~\cite{weak}, adapting them to our setting was not
completely straightforward.  Beyond the obvious modifications (more sparsifiers and a slightly more difficult
discharging argument), a major source of complications comes a bit surprisingly from the very last part of the proof dealing with non-contractible
triangles.  In~\cite{weak}, rather than considering the clustered coloring and bounding the maximum degree, we considered
a related notion of the weak diameter coloring (which coincides with clustered coloring on graphs with bounded maximum degree).
An advantage of this notion is that fixing a coloring of a bounded number of vertices does not affect the
weak diameter choosability, and consequently we got the final precoloring extension argument for free.
This is not the case for clustered coloring in general, as can be seen e.g. in Observation~\ref{obs-crucial},
forcing us to deal with the clustered precoloring-choosability rather than just with clustered choosability.
This in turn substantially complicates the rest of the argument.

\section{Islands}\label{sec-islands}

The starting point of our proof is a standard island argument.
A non-empty set $I\subseteq V(G)$ is a \emph{$c$-island} if every vertex in $I$ has less than $c$ neighbors outside of $I$.
For real numbers $a$ and $b$, we say that a graph $G$ is \emph{$(a,b)$-sparse} if $|E(G)|\le a|V(G)|+b$,
and \emph{hereditarily $(a,b)$-sparse} if every induced subgraph of $G$ is $(a,b)$-sparse.
The following claim is proved in~\cite{islands} (even if one assumes only the existence of strongly sublinear separators
rather than drawing on a fixed surface).
\begin{lemma}\label{lemma-islands}
For all integers $b$ and $c\ge 1$, for any real number $\varepsilon>0$,
and for every surface $\Sigma$, there exists a positive integer $\sigma$ such that the following claim holds:
Every $(c-\varepsilon,b)$-sparse graph $G$ drawn on $\Sigma$ contains
at least $|V(G)|/\sigma$ pairwise disjoint $c$-islands of size at most $\sigma$.
\end{lemma}

By a standard argument, the presence of small $c$-islands can be used to obtain clustered colorings.
We are going to need the following technical version of the statement.  For a graph $G$ and an integer $c$,
a set $Y\subseteq V(G)$ is \emph{$c$-solitary} if each vertex $v\in V(G)\setminus Y$ has less than $c$ neighbors in $Y$.
A coloring $\varphi$ \emph{isolates} a set $Y$ if $\varphi(u)\neq\varphi(v)$ holds for every edge $uv$ such that $u\in Y$ and $v\not\in Y$;
in other words, if every cluster intersecting $Y$ is contained in $Y$.

\begin{corollary}\label{cor-colin}
For any integer $c\ge 1$, any real number $\varepsilon>0$,
and every surface $\Sigma$, there exists a function $\gamma_{\ref{cor-colin}}:\mathbb{N}^2\to\mathbb{N}$ such that the following claim holds.
Let $b$ be a real number, let $G$ be graph drawn on $\Sigma$, let $L$ be a $c$-list-assignment
for $G$, and let $X$ be a set of vertices of $G$.
If $G$ is hereditarily $(c-\varepsilon, b)$-sparse, then every $L$-coloring $\psi$ of $X$ extends
to an $L$-coloring $\varphi$ of $G$ with clustering at most $\gamma_{\ref{cor-colin}}(b,|X|)$
such that if $X$ is $c$-solitary, then $\varphi$ isolates $X$.
\end{corollary}
\begin{proof}
For any real number $b$, let $\sigma(b)$ be the constant from Lemma~\ref{lemma-islands} for $b$, $c$, $\varepsilon$, and $\Sigma$,
and let $\gamma_{\ref{cor-colin}}(b,p)=\sigma(b)\cdot p$.
We prove the claim by induction on the number of vertices of $G$.  If $G$ has more than $\sigma(b)\cdot |X|$ vertices,
then by Lemma~\ref{lemma-islands}, $G$ contains more than $|X|$ $c$-islands of size at most $\sigma(b)$.
Hence $G$ contains a $c$-island $I$ of size at most $\sigma(b)$ disjoint from $X$.
By the induction hypothesis, $\psi$ extends to an $L$-coloring $\varphi_0$ of $G-I$
with clustering at most $\gamma_{\ref{cor-colin}}(b,|X|)$ such that if $X$ is $c$-solitary, then $\varphi_0$ isolates $X$.
We extend $\varphi_0$ to an $L$-coloring $\varphi$ of $G$ by choosing for each vertex $v\in I$ an arbitrary color from $L(v)$ different from the colors of
its (less than $c$) neighbors outside of $I$.  This ensures that all newly arising clusters are contained in $I$,
and thus they have size at most $\sigma(b)$. Moreover, if $\varphi_0$ isolates $X$, then so does $\varphi$.

Suppose now that $|V(G)|\le \sigma(b)\cdot |X|=\gamma_{\ref{cor-colin}}(b,|X|)$.
If $X$ is $c$-solitary and $X\neq V(G)$, then we proceed in the same way for the
$c$-island $I=V(G)\setminus X$.  Otherwise, we can simply extend $\psi$ to an $L$-coloring of $G$ arbitrarily.
\end{proof}

\section{Sparsifiers}\label{sec-sparsifiers}

The generalized Euler's formula implies that a graph $G$ drawn on a surface of Euler genus $g$ is hereditarily $(3,O(g))$-sparse.
This is unfortunately just too dense for us to be able to apply Corollary~\ref{cor-colin} with $c=3$, since there
we need the graph to be hereditarily $(3-\varepsilon,O(g))$-sparse for fixed $\varepsilon>0$.  To circumvent
this issue, let us introduce a notion of \emph{sparsifiers}---subgraphs of $G$ with the property that
\begin{itemize}
\item deleting a maximal system of disjoint non-adjacent sparsifier results in a hereditarily $(3-\varepsilon,O(g))$-sparse graph $G'$, and
\item a coloring of $G'$ with clustering at most $\gamma$ extends to a coloring of $G$ (from the given lists) with clustering $O(\gamma)$.
\end{itemize}
An important issue is how to ensure hereditarity; it is fairly easy to come up with a set of sparsifiers whose deletion results in a $(3-\varepsilon,O(g))$-sparse graph,
but how to ensure that all induced subgraphs $H\subseteq G'$ are also sparse?  To deal with this issue, we find sparsifiers with the property
that all incident faces have length three, and restrict ourselves for now to graphs without separating triangles.
This way, any sparsifier $S$ in the induced subgraph $H$ would also be a sparsifier in $G'$:  The vertices of $S$ are incident only with triangular faces in $H$,
and since $G$ does not contain separating triangles, these triangles also bound the same faces in $G'$.  Consequently, vertices of $S$ have the same neighborhoods
in $G'$ as in $H$, and as will be clear when we define sparsifiers formally, this is enough to ensure that $S$ is a sparsifier in $G'$ as well.
Let us now present the ideas precisely, including the additional complications arising from dealing with precolored vertices.

For an integer $D$, a graph $G$ drawn on a surface, and a set $X\subseteq V(G)$,
a vertex $v\in V(G)$ is \emph{$(D,X)$-small} if $v\not\in X$, $v$ has degree at most $D$,
and all incident faces are $2$-cell and have length exactly three.  We say that $v$ is \emph{$(D,X)$-big} otherwise (i.e., $v\in X$, or $\deg v>D$,
or $v$ is incident with a non-triangular face).
Note that in discharging arguments, it is more usual to define a big vertex to be simply a vertex of sufficiently
large degree.  The reason for our more complicated definition is to make this notion hereditary: Subject to
a technical constraint on separating triangles, if a vertex is $(D,X)$-big in a graph, it is also $(D,X)$-big in all induced subgraphs that contain it.
We say that a triangle in $G$ is \emph{facial} if it bounds a $2$-cell face, and \emph{$X$-external} if it contains a vertex
that does not belong to $X$.  Let us remark that all graphs considered in this paper are simple, without loops or parallel edges.

\begin{observation}\label{obs-sg}
Suppose $G$ is a graph drawn on a surface and let $X$ be a set of vertices of $G$ such that every $X$-external triangle in $G$ is facial.
For any integer $D$, any subgraph $H$ of $G$ with $|V(H)|>3$, and a vertex $v\in V(H)$, if $v$ is $(D,X\cap V(H))$-small in $H$, then
$v$ is also $(D,X)$-small in $G$ and $\deg_H v=\deg_G v$.
\end{observation}
\begin{proof}
Since $v$ is $(D,X\cap V(H))$-small in $H$, it follows that $v\not\in X$ and all faces of $H$ incident with $v$ are $2$-cell and bounded by triangles.
Since $H$ is a simple graph and $|V(H)|>3$, this implies that $\deg_H v\ge 3$, and thus each of the triangles bounds exactly one face of $H$.
Since $v\not\in X$, the triangles containing $v$ are $X$-external, and thus by the assumptions, they are also facial in $G$.
We conclude that each face of $H$ incident with $v$ is also a face of $G$.  It follows that $\deg_G v=\deg_H v\le D$
and that $v$ is $(D,X)$-small in $G$.
\end{proof}

\begin{figure}
\begin{center}
\includegraphics{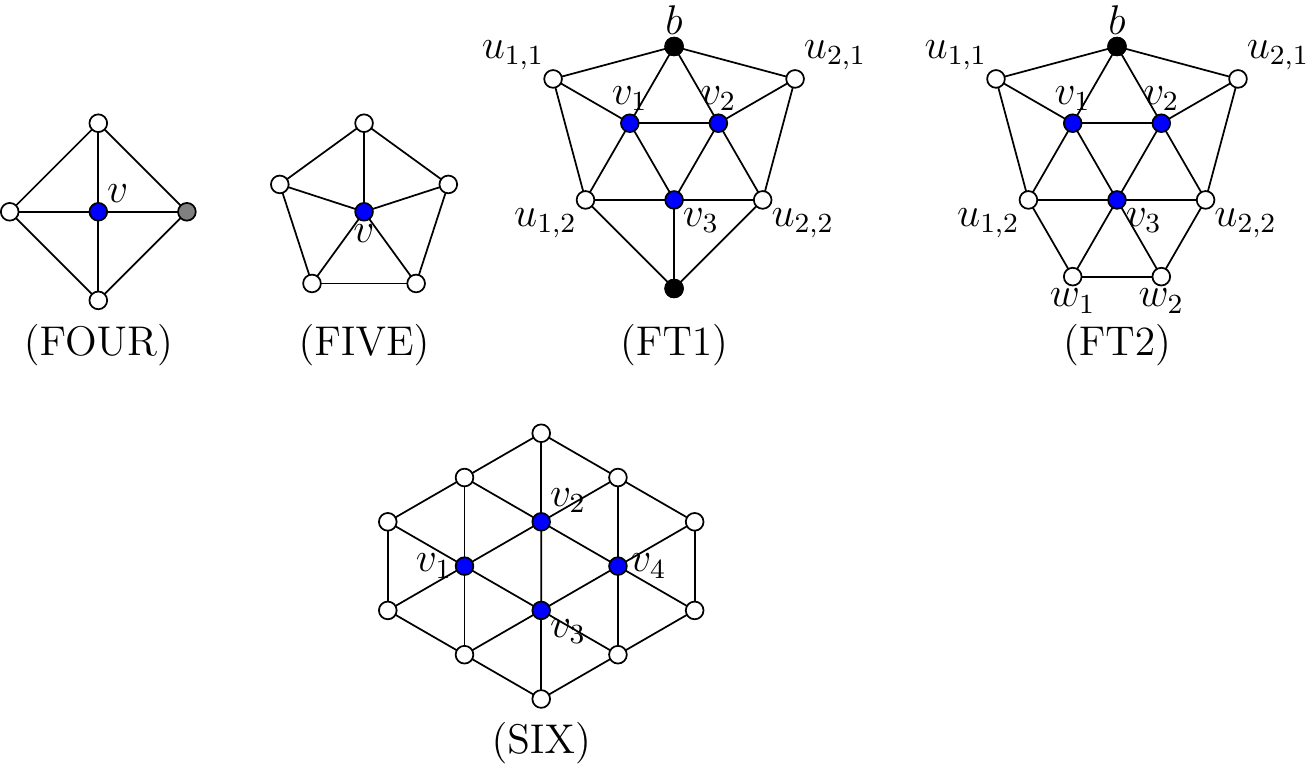}
\end{center}
\caption{The sparsifiers. White vertices are $(D,X)$-small, black are $(D,X)$-big, gray may be either.  The $(D,X)$-small vertices forming the sparsifier
are depicted in blue.}\label{fig-sparsifiers}
\end{figure}

For an integer $D\ge 6$, a graph $G$ drawn on a surface, and a set $X\subseteq V(G)$,
a set $S\subseteq V(G)$ is a \emph{$(D,X)$-sparsifier} if it consists of $(D,X)$-small vertices satisfying one of the following conditions
(see Figure~\ref{fig-sparsifiers}):
\begin{itemize}
\item[(FOUR)] $S=\{v\}$, where $v$ is a vertex of degree four with at most one $(D,X)$-big neighbor.
\item[(FIVE)] $S=\{v\}$, where $v$ is a vertex of degree five with no $(D,X)$-big neighbors.
\item[(FT1)] $S=\{v_1,v_2,v_3\}$, where $v_1v_2v_3$ is a facial triangle of vertices of degree five,
each with exactly one $(D,X)$-big neighbor, and such that $v_1$ and $v_2$ have a common $(D,X)$-big neighbor.
\item[(FT2)] $S=\{v_1,v_2,v_3\}$, where $v_1v_2v_3$ is a facial triangle, $v_1$ and $v_2$ are degree five vertices
with a common $(D,X)$-big neighbor and no other $(D,X)$-big neighbors and $v_3$ is a vertex of degree six with no $(D,X)$-big neighbor.
\item[(SIX)] $S=\{v_1,v_2,v_3,v_4\}$, where $v_1v_2v_3$ and $v_2v_3v_4$ are facial triangles sharing an edge and consisting
of vertices of degree six with no $(D,X)$-big neighbors, and $v_1$ is not adjacent to $v_4$.
\end{itemize}

Two sets $S_1,S_2\subseteq V(G)$ are \emph{separated} if they are vertex-disjoint and $G$ has no edge with one end in $S_1$
and the other end in $S_2$.  A graph $G$ is \emph{$(D,X)$-sparsifier-free} if it contains no $(D,X)$-sparsifiers.
Let us note the following important corollary of Observation~\ref{obs-sg}.

\begin{lemma}\label{lemma-nonewspars}
Let $D\ge 6$ be an integer.  Suppose $G$ is a graph drawn on a surface and $X$ is a set of vertices of $G$ such
that every $X$-external triangle in $G$ is facial.
Let $S_1$, \ldots, $S_m$ be a maximal system of pairwise separated $(D,X)$-sparsifiers in $G$.
Then every induced subgraph $H$ of $G-\bigcup_{i=1}^m S_i$ is $(D,X\cap V(H))$-sparsifier-free.
\end{lemma}
\begin{proof}
Suppose for a contradiction that $S$ is a $(D,X\cap V(H))$-sparsifier in $H$.  By the definition of sparsifiers,
vertices of $S$ are $(D,X\cap V(H))$-small in $H$, and $|V(H)|\ge 5$.  By Observation~\ref{obs-sg}, the vertices of $S$ are also $(D,X)$-small in $G$
and have the same neighborhood in $G$ as in $H$.  Consequently, $S$ is separated from $S_1$, \ldots, $S_m$.
Moreover, by Observation~\ref{obs-sg}, all $(D,X)$-big vertices in $V(G)\setminus S$ with a neighbor in $S$
are also $(D,X\cap V(H))$-big in $H$.  The inspection of the definition of sparsifiers shows that either
$S$ or a subset of $S$ is a $(D,X)$-sparsifier in $G$ (the subset case happens if $S$ satisfies (FT1) or (FT2)
and one of the neighboring vertices is $(D,X\cap V(H))$-big in $H$ but not $(D,X)$-big in $G$; then a vertex of $S$
of degree five forms a $(D,X)$-sparsifier in $G$ satisfying (FIVE)).
However, this contradicts the maximality of the system $S_1$, \ldots, $S_m$.
\end{proof}

The choice of the subgraphs forming the sparsifiers is motivated by the following coloring extension property.
\begin{lemma}\label{lemma-delspars}
Let $G$ be a graph drawn on a surface and $X$ a set of its vertices, let $D\ge 6$ be an integer,
let $L$ be a $3$-list-assignment for $G$, and let $S_1$, \ldots, $S_m$ be pairwise separated $(D,X)$-sparsifiers in $G$.
Then any $L$-coloring $\varphi_0$ of $G'=G-\bigcup_{i=1}^m S_i$ with clustering at most $k$ extends to
an $L$-coloring $\varphi$ of $G$ with clustering at most $(4D+1)k$.  Moreover, if $\varphi_0$ isolates $X$, then so does~$\varphi$.
\end{lemma}
\begin{proof}
Consider any $i\in\{1,\ldots,m\}$, let $H_i$ be the subgraph of $G$ induced by the neighbors of vertices of $S_i$
in $V(G)\setminus S_i$, and let $T_i$ be the subgraph induced by the vertices in $S_i$ and their neighbors.
Note that $H_i\subseteq G'$, since the sparsifiers $S_1$, \ldots, $S_m$ are pairwise separated.
Below, we show that $\varphi_0$ can be extended to each sparsifier $S_i$ so that the following conditions hold:
\begin{itemize}
\item[(i)] If two vertices of $H_i$ are connected by a monochromatic path in $T_i$, then they are also connected
by a monochromatic path in $H_i$.
\item[(ii)] The color of each vertex of $S_i$ is different from the color of any adjacent $(D,X)$-big neighbor.
\end{itemize}
Let $\varphi$ be obtained from $\varphi_0$ by extending it to each of the sparsifiers separately so that the conditions
(i) and (ii) hold for each of them.  The condition (i) ensures that each cluster $C$ of $\varphi$ is either
contained in one of the sparsifiers (and thus has size at most $4$), or it is obtained from a cluster $C_0$ of $\varphi_0$
by adding parts of sparsifiers with neighbors in $C_0$ (i.e., no two previously existing clusters are merged).
The condition (ii) implies that each component of $G[C]-C_0$ has a neighbor in $C_0$ that is $(D,X)$-small,
and thus there are at most $D|C_0|$ such components.  Therefore, $|C|\le |C_0|+4D|C_0|\le (4D+1)k$.
Moreover, (ii) also ensures that if $\varphi_0$ isolates $X$, then so does $\varphi$.

It remains to show that $\varphi_0$ extends to an $L$-coloring of each $(D,X)$-sparsifier $S_i$ so that
the conditions (i) and (ii) hold.  
For a vertex $v\in S_i$, let us say that a color $c\in L(v)$ is \emph{dangerous} if either
$c$ appears on distinct neighbors of $v$ in $H_i$ that are not connected by a monochromatic path in $H_i$,
or $c$ appears on a $(D,X)$-big neighbor of $v$ in $H_i$.  Note that a dangerous color cannot be used to color $v$,
as otherwise (i) and (ii) would be violated.  If each vertex in $S_i$ is colored by a non-dangerous color, then (ii)
holds, but we still need to be careful in order to satisfy (i) in case that the same color is used on more than one vertex
of $S_i$.  Let us now consider each type of sparsifier separately.
\begin{itemize}
\item[(FOUR)] Let $S_i=\{v\}$.  Since $v$ has degree four and has at most one $(D,X)$-big neighbor, there
are at most two dangerous colors in $L(v)$, and coloring $v$ by a non-dangerous color from $L(v)$
ensures that both (i) and (ii) hold.
\item[(FIVE)] The argument in this case is identical, noticing that since $v$ has degree five and no $(D,X)$-big neighbors,
there are at most two dangerous colors in $L(v)$.
\item[(FT1)] Let $S_i=\{v_1,v_2,v_3\}$, where $v_1$ and $v_2$ have a common $(D,X)$-big neighbor $b$.
Note that there is only one dangerous color at $v_1$ and $v_2$, since for $j\in\{1,2\}$, the neighbors
$u_{j,1}$ and $u_{j,2}$ of $v_j$ outside of $S_i$ distinct from $b$ are adjacent,
whereas there can be up to two dangerous colors at $v_3$.  Choose the labels so that $u_{1,2}$ and $u_{2,2}$ are
also adjacent to $v_3$, see Figure~\ref{fig-sparsifiers}.  Let us first
try coloring $v_1$ and $v_2$ by non-dangerous colors $c_1\in L(v_1)$ and $c_2\in L(v_2)$ with $c_1\neq c_2$
and $v_3$ by a non-dangerous color $c_3\in L(v_3)$; this clearly ensures that (ii) holds.
If $c_3\not\in\{c_1,c_2\}$, then since we selected a non-dangerous color at each vertex, 
the coloring also satisfies (i).  Hence, by symmetry we can assume that $c_3=c_1$.
The condition (i) can only be violated if $\varphi_0(u_{2,2})=\varphi_0(u_{1,1})=c_1$ and $\varphi_0(u_{1,2})\neq c_1$
(if $\varphi_0(u_{1,2})=c_1$, then since $c_3=c_1$ is not dangerous at $v_3$, $u_{2,2}$ and $u_{1,2}$ are contained
in the same cluster of $H_i$, and so is the vertex $u_{1,1}$ adjacent to $u_{2,2}$).

However, since there are two non-dangerous colors at $v_1$ and $v_2$, there exists another choice of non-dangerous
colors $c'_1\in L(v_1)$ and $c'_2\in L(v_2)$ such that $c'_2\neq c'_1\neq c_1$.  We still color $v_3$ by $c_1$.
Then $v_1$ has color different from the colors of its neighbors in $S_i$.  Moreover, if $v_2$ and $v_3$ receive the same color,
then this color is $c_1\neq\varphi(u_{1,2})$.  We conclude that this coloring satisfies (i) and (ii).

\item[(FT2)] Let $S_i=\{v_1,v_2,v_3\}$, where $v_1$ and $v_2$ are degree five vertices with a common $(D,X)$-big neighbor $b$.
For $j\in\{1,2\}$, let $u_{j,1}$ and $u_{j,2}$ be the neighbors of $v_j$ outside of $S_i$ distinct from $b$,
where $u_{1,2}$ and $u_{2,2}$ are adjacent to $v_3$, and let $w_1$ and $w_2$ be the other two neighbors of $v_3$
outside of $S_i$, where $w_1$ is adjacent to $u_{1,2}$, see Figure~\ref{fig-sparsifiers}.  As in the previous
case, note that there is only one dangerous color at $v_1$ and $v_2$ and up to two dangerous colors at $v_3$.
Let us first try coloring $v_1$ and $v_2$ by non-dangerous colors $c_1\in L(v_1)$ and $c_2\in L(v_2)$ with $c_1\neq c_2$
and $v_3$ by a non-dangerous color $c_3\in L(v_3)$, so that (ii) holds.
If $c_3\not\in\{c_1,c_2\}$, then since we selected a non-dangerous color at each vertex, 
the coloring also satisfies (i).  Hence, by symmetry we can assume that $c_3=c_1$.
The condition (i) can only be violated if $\varphi_0(u_{1,1})=c_1$,
$c_1\in\{\varphi_0(w_1),\varphi_0(w_2),\varphi_0(u_{2,2})\}$, and $\varphi_0(u_{1,2})\neq c_1$
(if $\varphi_0(u_{1,2})=c_1$, then since $c_3=c_1$ is not dangerous at $v_3$, all neighbors of $v_3$ of color $c_1$ in $H_i$
are contained in the same cluster of $H_i$, and so is the neighbor $u_{1,1}$ of $u_{1,2}$).

Next, let us try another choice of non-dangerous colors $c'_1\in L(v_1)$ and $c'_2\in L(v_2)$ such that $c'_2\neq c'_1\neq c_1$,
while still coloring $v_3$ by $c_1$; by a symmetric argument, (ii) holds, and (i) is satisfied unless $c'_2=c_1$,
$\varphi_0(u_{2,1})=c_1$, $\varphi_0(u_{2,2})\neq c_1$, and $c_1\in \{\varphi_0(w_1),\varphi_0(w_2))\}$ (the case $\varphi_0(u_{1,2})=c_1$
was excluded in the previous paragraph).  Since $c_1\in \{\varphi_0(w_1),\varphi_0(w_2))\}$ and $\varphi_0(u_{1,2})\neq c_1\neq \varphi_0(u_{2,2})$,
there is at most one dangerous color at $v_3$.  Hence, we can extend $\varphi_0$ to an $L$-coloring of $S_i$
satisfying (i) and (ii) by giving $v_1$ color $c_1$, $v_2$ color $c_2$, and $v_3$ a non-dangerous color different from $c_1$.
\item[(SIX)] Let $S_i=\{v_1,v_2,v_3,v_4\}$, where $v_1$ is not adjacent to $v_4$.
For $u\in S_i$, let $B(u)$ be the set of colors that $\varphi_0$ assigns to neighbors of $u$.
Let $R=\{u\in S_i:|B(u)|\le 2\}$.  For each $u\in R$, let $\varphi(u)\in L(u)\setminus B(u)$ be chosen arbitrarily;
the color $\varphi(u)$ clearly is not dangerous at $u$.

Consider now a vertex $u\in S_i\setminus R$ and let $L'(u)$ consist of the colors in $L(u)$ that are not dangerous.
Note that if $u\in \{v_1,v_4\}$, then since $|B(u)|\ge 3$, there is at most one dangerous color at $u$ and $|L'(u)|\ge 2$.
For $u\in \{v_2,v_3\}$, $|B(u)|\ge 3$ means that there are no dangerous colors at $u$ and $|L'(u)|\ge 3$.
Consequently $|L'(u)|\ge \deg_{G[S_i]} u$ for each $u\in V(S_i)\setminus R$.  Since $G[S_i]$ is connected and not a Gallai tree,
this implies that we can choose $\varphi$ on $V(S_i)\setminus R$ to be a proper $L'$-coloring of $G[S_i\setminus R]$.
Additionally, in case that $v_1,v_4\not\in R$ and $\{v_2,v_3\}\cap R\neq\emptyset$
(so $G[S_i\setminus R]$ either is a path or consists of two isolated vertices), observe that we can choose $\varphi$ so that
$\varphi(v_1)\neq\varphi(v_4)$.

Since we colored each vertex of $S_i$ by a non-dangerous color, the condition (i) holds unless
distinct vertices $x,y\in V(S_i)$ both receive the same color $c$ and this color belongs to $B(x)\cap B(y)$.
Note that $x,y\not\in R$, since each vertex $u\in R$ was colored by a color not in $B(u)$.
Moreover, since $\varphi$ is a proper coloring on $G[S_i\setminus R]$, we have $xy\not\in E(G[S_i])$, and thus $\{x,y\}=\{v_1,v_4\}$.
By the last condition in the choice of $\varphi$, since $\varphi(v_1)=\varphi(v_4)$, we have $R=\emptyset$.
Consequently, no other vertex of $S_i$ has color $c$, and $G[S_i]$ does not contain a monochromatic path
between $x$ and $y$.  Therefore, the condition (i) is satisfied.  The condition (ii) clearly holds,
since the vertices in $S_i$ have no $(D,X)$-big neighbors.
\end{itemize}
\end{proof}

\section{Discharging}\label{sec-disch}

By Lemmas~\ref{lemma-nonewspars} and \ref{lemma-delspars}, we can eliminate all sparsifiers
while not increasing the clustering too much (importantly, Lemma~\ref{lemma-nonewspars} implies that
the absence of sparsifiers is preserved in induced subgraphs, and thus we only need to eliminate
them once).  As we show next, the graphs without sparsifiers are so sparse that we can apply Corollary~\ref{cor-colin} for them.

\begin{lemma}\label{lemma-spaspar}
Let $t\ge 2$ be an integer, let $G$ be a $P''_t$-free graph drawn on a surface of Euler genus $g$ 
and let $X$ be a set of vertices of $G$ such that every $X$-external triangle is facial.
If $G$ is $(336t,X)$-sparsifier-free, then $G$ is $\bigl(3-\frac{1}{208t},4g+|X|+t+2\bigr)$-sparse.
\end{lemma}
\begin{proof}
Let $D=336t$.  We can assume $|V(G)|>t+2\ge 4$, as otherwise the claim holds trivially.
We can also without loss of generality assume that no component of $G$ has at most two vertices,
as deleting such components only makes establishing the bound on the number of edges of $G$ harder.
In particular, every face of $G$ has length at least three.

Let $\beta=\sum_f (|f|-3)$, where the sum is over all faces $f$ of $G$,
and let $m$ be the number of faces of $G$.
By generalized Euler's formula, we have $|E(G)|\le |V(G)|+m+g-2$, and since $2|E(G)|=\sum_f |f|=3m+\beta$,
we conclude that
\begin{equation}\label{eq-dens}
|E(G)|<3(|V(G)|+g)-\beta.
\end{equation}
Note that at most $\sum_{f:|f|\ge 4} |f|\le \sum_f 4(|f|-3)\le 4\beta$ vertices of $G$ are incident with a face
of length at least four, and at most $\frac{2|E(G)|}{D}<\frac{6(|V(G)|+g)}{D}$ vertices of $G$ have degree more than $D$.
Consequently, the number $\theta$ of $(D,X)$-big vertices of $G$ satisfies
\begin{equation}\label{eq-numbig}
\theta\le 4\beta+\frac{6(|V(G)|+g)}{D}+|X|.
\end{equation}

Let us assign the initial charge to vertices of $G$ as follows:
\begin{itemize}
\item Each $(D,X)$-small vertex $v$ receives initial charge $\ch_0(v)=\deg v-6$.
\item Each $(D,X)$-big vertex $v$ receives initial charge $\ch_0(v)=\deg v+1$.
\end{itemize}
By (\ref{eq-dens}) and (\ref{eq-numbig}), the sum of the initial charges is
\begin{align}
\sum_{v\in V(G)} \ch_0(v)&=\Bigl(\sum_{v\in V(G)} \deg v-6\Bigr)+7\theta=2|E(G)|-6|V(G)|+7\theta\nonumber\\
&<6g-2\beta+7\Bigl(4\beta+\frac{6(|V(G)|+g)}{D}+|X|\Bigr)\nonumber\\
&\le 7g+7|X|+26\beta+\frac{|V(G)|}{8t}.\label{eq-initial}
\end{align}
We say a vertex $v$ is \emph{safe} if $v$ is $(D,X)$-small and
\begin{itemize}
\item $\deg v=4$ and $v$ is contained in a facial triangle $vu_1u_2$ such that the vertices $u_1$ and $u_2$ are $(D,X)$-big, or
\item $\deg v=5$ and $v$ has at least two $(D,X)$-big neighbors, or
\item $\deg v=6$ and $v$ has a $(D,X)$-big neighbor, or
\item $\deg v\ge 7$.
\end{itemize}
We now redistribute the charge according to the following rules (see Figure~\ref{fig-rules}):

\begin{figure}
\begin{center}
\includegraphics{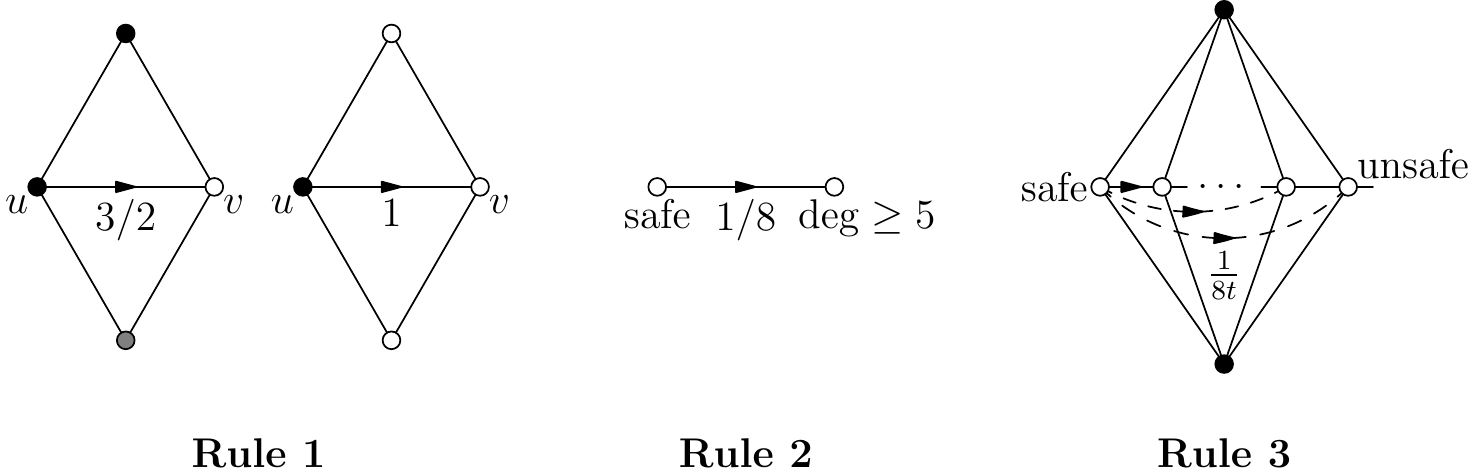}
\end{center}
\caption{The discharging rules. White vertices are $(D,X)$-small, black are $(D,X)$-big, gray may be either.}\label{fig-rules}
\end{figure}

\begin{description}
\item[Rule 1:] Suppose $uv\in E(G)$, $u$ is $(D,X)$-big and $v$ is $(D,X)$-small.
If there exists a facial triangle $uvz$ such that the vertex $z$ is $(D,X)$-big, then $u$ sends $3/2$ to $v$,
otherwise $v$ sends $1$ to $v$.
\item[Rule 2:] Each safe vertex sends $1/8$ to each of its $(D,X)$-small neighbors
of degree at least five.
\item[Rule 3:] If $P$ is a path in $G$, the first vertex $u$ of $P$ is safe,
all other vertices of $P$ are unsafe $(D,X)$-small vertices of degree four,
and all vertices of $P$ have two common $(D,X)$-big neighbors,
then $u$ sends $\frac{1}{4t}$ to the last vertex of $P$.
\end{description}
We claim that the final charge $\ch(v)$ of each vertex $v$ of $G$ after the charge redistribution is at least $\tfrac{1}{4t}$.
\begin{itemize}
\item If $v$ is $(D,X)$-big, then $v$ only sends charge by Rule 1.  If $v$
has $s$ $(D,X)$-small and $b$ $(D,X)$-big neighbors, then it sends
at most $(s-2b)\times 1+2b\times 3/2=s+b=\deg v$ units of charge.  Hence, its final charge
is at least $\ch_0(v)-\deg v = 1>\tfrac{1}{4t}$.

\item If $v$ is safe, then an inspection of the definition of a safe vertex shows that the sum of its initial charge and the amount received
by Rule 1 is at least $\max(\deg v-6,1)$.  For each adjacent unsafe vertex $z$,
if $z$ has degree at least five then $v$ sends $1/8$ to $z$ by Rule 2.
If $z$ has degree four, then since $G$ is $P''_t$-free, the longest path
starting with the edge $vz$ and satisfying the conditions of Rule 3 has at most $t-1$
vertices, and thus $v$ sends at most $(t-2)\times \frac{1}{8t}<\tfrac{1}{8}$ using Rule 3 to the vertices
of this path.  Therefore,
$$\ch(v)\ge \max(\deg v-6,1)-\frac{1}{8}\deg v\ge \frac{1}{8}\ge \frac{1}{4t}.$$

\item Suppose $v$ is an unsafe $(D,X)$-small vertex of degree six.  Since $v$ is not safe, all its neighbors are $(D,X)$-small.
If any of them is safe, then $v$ receives $1/8$ by Rule 2 and its final charge is at least $\ch_0(v)+\tfrac{1}{8}=\tfrac{1}{8}\ge\tfrac{1}{4t}$.

Hence, suppose that all neighbors of $v$ are unsafe.  Since all faces incident with $v$ are 2-cell faces of length three and $G$ is a simple graph,
the neighbors of $v$ form a cycle $C$.  Since every $X$-external triangle in $G$ is facial, the cycle $C$ is induced, and since all
faces incident with vertices of $C$ have length three, this implies that each vertex of $C$ has degree at least four in $G$.
If any neighbor of $v$ had degree four, then it would have at most one
$(D,X)$-big neighbor, and thus it would form a $(D,X)$-sparsifier satisfying (FOUR); hence, we can
assume all neighbors of $v$ have degree five or six.

If $v$ has a neighbor $u_1$
of degree five, then since $\{u_1\}$ is not a $(D,X)$-sparsifier satisfying (FIVE), $u_1$ has a $(D,X)$-big neighbor $x$.
Let $u_2$ be the common neighbor of $x$, $u_1$ and $v$.  Since $u_2$ is not safe, it also has degree five.
However, then $\{u_1,u_2,v\}$ is a $(D,X)$-sparsifier satisfying (FT2).  Finally, if all neighbors of $v$
have degree six, then $v$ together with three consecutive neighbors forms a $(D,X)$-sparsifier satisfying (SIX).
In all the cases we obtain a contradiction, since $G$ is $(D,X)$-sparsifier-free.
\item Suppose $v$ is an unsafe $(D,X)$-small vertex of degree five, and let $u_1\ldots u_5$ be the cycle formed by the neighbors
of $v$.  Since $G$ does not contain a $(D,X)$-sparsifier satisfying (FIVE), we can assume that $u_1$ is $(D,X)$-big,
and since $v$ is not safe, $u_2$, \ldots, $u_5$ are $(D,X)$-small.  If any of them is safe, then $v$ receives
$1$ from $u_1$ by Rule 1 and additional $1/8$ by Rule 2, resulting in final charge at least $\ch_0(v)+1+\tfrac{1}{8}=\tfrac{1}{8}\ge \tfrac{1}{4t}$.

Suppose that $u_2$, \ldots, $u_5$ are unsafe.
If $\deg u_i=4$ for some $i\in \{2,\ldots,5\}$, then note that $u_i$ has at most one $(D,X)$-big neighbor
(this is clear for $i\in \{3,4\}$, and for $i\in\{2,5\}$, the second $(D,X)$-big neighbor would be adjacent to $u_1$,
making $u_i$ safe).  Hence $\{u_i\}$ would be a $(D,X)$-sparsifier satisfying (FOUR), which is a contradiction.
Since $u_2$ is unsafe, it follows that $\deg u_2=5$ and $u_1$ is its only $(D,X)$-big neighbor.  If $\deg u_3=5$
and all neighbors of $u_3$ are $(D,X)$-small, then $\{u_3\}$ is a $(D,X)$-sparsifier satisfying (FIVE).
If $\deg u_3=5$ and $u_3$ has a $(D,X)$-big neighbor, then $\{v,u_2,u_3\}$ is a $(D,X)$-sparsifier satisfying (FT1).
Finally, if $\deg u_3=6$, then since $u_3$ is unsafe, it has no $(D,X)$-big neighbor and 
$\{v,u_2,u_3\}$ is a $(D,X)$-sparsifier satisfying (FT2).  In all cases, we obtain a contradiction.

\item Suppose that $v$ is an unsafe $(D,X)$-small vertex of degree four.  Since $(G,X)$ does not contain a $(D,X)$-sparsifier
satisfying (FOUR), $v$ has two $(D,X)$-big neighbors $x$ and $y$, and since $v$ is unsafe, $x$ and $y$ are not consecutive
in the cycle induced by the neighbors of $v$.
Let $P$ be a maximal path of unsafe $(D,X)$-small vertices of degree four containing $v$ such that each vertex of $P$ is adjacent to $x$
and $y$.  Since $G$ is $P''_t$-free, $P$ has at most $t-1$ vertices.  Consider an end $v'$ of $P$.
If $|V(P)|\ge 3$ and $v'$ is adjacent to the other end of $P$,
then since every $X$-external triangle in $G$ is facial, we conclude that $V(G)=V(P)\cup \{x,y\}$.  This is a contradiction,
since we have assumed that $|V(G)|>t+1$.  Therefore, $v'$ has a neighbor $u$ outside of $V(P)\cup \{x,y\}$.
This neighbor is $(D,X)$-small since $v'$ is unsafe, and safe since it has two $(D,X)$-big neighbors $x$ and $y$
(if $\deg u=4$, its safety follows from the maximality of $P$).  Consequently $u$ sends $\frac{1}{8t}$ to $v$ by Rule 3,
and so does the neighbor of the other end of $P$.  Additionally, $u$ receives $1$ from each of $x$ and $y$ by Rule 3.
Therefore, the final charge of $v$ is at least $\ch_0(v)+2\times 1+2\times\frac{1}{8t}=\frac{1}{4t}$.

\item Finally, suppose $v$ is an (unsafe) $(D,X)$-small vertex of degree at most three.
Since $G$ is simple, $v$ has degree exactly three (if $v$ had degree two,
then since $v$ is only incident with 2-cell faces of length three, either $G$ would
contain a double edge joining the neighbors of $v$, or $G=K_3$ would have less than four vertices).
The neighbors of $v$ form a triangle, and this triangle cannot be $X$-external, since otherwise it would be facial
and $G=K_4$ would have at most four vertices.
Hence, all neighbors of $v$ are $(D,X)$-big, and $v$ receives $3/2$ from each of them by Rule 1.
It follows that the final charge of $v$ is $\ch_0(v)+3\times \tfrac{3}{2}=\tfrac{3}{2}>\tfrac{1}{4t}$.

\end{itemize}
In conclusion, the final charge of every vertex is indeed at least $\frac{1}{4t}$.  Since no charge is created or lost,
(\ref{eq-initial}) implies that
\begin{align*}
\frac{|V(G)|}{4t}&\le \sum_{v\in V(G)} \ch(v)=\sum_{v\in V(G)} \ch_0(v)\\
&\le 7g+7|X|+26\beta+\frac{|V(G)|}{8t}.
\end{align*}
It follows that
$$\beta\ge \frac{|V(G)|}{208t}-g-|X|.$$
Substituting into (\ref{eq-dens}), we obtain
$$|E(G)|<\Bigl(3-\frac{1}{208t}\Bigr)|V(G)|+4g+|X|,$$
implying that $G$ is $\bigl(\frac{1}{208t},4g+|X|+t+2\bigr)$-sparse.
\end{proof}

Let us now combine these claims.

\begin{corollary}\label{cor-nocut}
For every surface $\Sigma$ and integer $t$, there exists a function $\gamma_{\ref{cor-nocut}}:\mathbb{N}\to\mathbb{N}$
such that the following claim holds.
Let $G$ be a graph drawn on $\Sigma$ and let $X$ be a set of vertices of $G$ such that every $X$-avoiding triangle is facial.
Let $L$ be a $3$-list-assignment for $G$.  If $G$ is $P''_t$-free,
then every $L$-coloring $\psi$ of $X$ extends to an $L$-coloring $\varphi$ of $G$ with clustering at most $\gamma_{\ref{cor-nocut}}(|X|)$
such that if $X$ is $3$-solitary, then $\varphi$ isolates $X$.
\end{corollary}
\begin{proof}
Without loss of generality, we can assume $t\ge 2$.
Let $\gamma_{\ref{cor-colin}}:\mathbb{N}^2\to\mathbb{N}$ be the function from Corollary~\ref{cor-colin}
for $c=3$, $\varepsilon=\frac{1}{208t}$ and $\Sigma$.  Let $D=336t$, let $G$ be the Euler genus of $\Sigma$ and let
$$\gamma_{\ref{cor-nocut}}(p)=(4D+1)\gamma_{\ref{cor-colin}}(4g+p+t+2,p).$$
Let $b=4g+|X|+t+2$.
Let $S_1$, \ldots, $S_m$ be a maximal system of pairwise separated $(D,X)$-sparsifiers in $G$.
Let $G_0=G-\bigcup_{i=1}^m S_i$.  By Lemma~\ref{lemma-nonewspars}, every induced subgraph $H$ of $G_0$
is $(D,X\cap V(H))$-sparsifier-free, and by Lemma~\ref{lemma-spaspar}, $H$ is $(3-\varepsilon,4g+|X\cap V(H)|+t+2)$-sparse,
and thus also $(3-\varepsilon, b)$-sparse.  Therefore, $G_0$ is hereditarily $(3-\varepsilon,b)$-sparse.
By Corollary~\ref{cor-colin}, $\psi$ extends to an $L$-coloring $\varphi_0$ of $G_0$ with clustering
at most $\gamma_{\ref{cor-colin}}(b,|X|)$ such that if $X$ is $3$-solitary, then $\varphi_0$ isolates $X$.
By Lemma~\ref{lemma-delspars}, $\varphi_0$ further extends to an $L$-coloring $\varphi$ of $G$ with clustering
at most $(4D+1)\gamma_{\ref{cor-colin}}(b,|X|)=\gamma_{\ref{cor-nocut}}(|X|)$
such that if $X$ is $3$-solitary, then $\varphi$ isolates $X$.
\end{proof}

\section{Planar 3-trees}\label{sec-stacks}

\begin{figure}
\begin{center}
\includegraphics{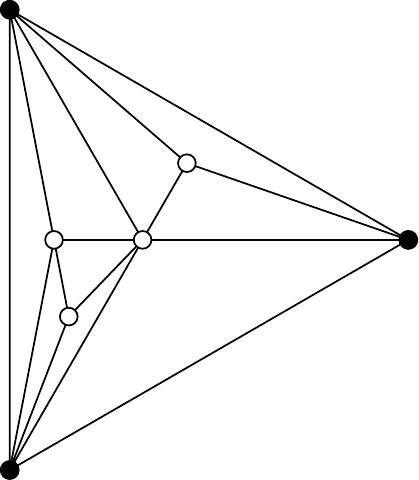}
\end{center}
\caption{A rooted planar 3-tree.}\label{fig-stacks}
\end{figure}

Next, we need to take care of non-facial $X$-external triangles.  Let us first consider an important special case.
A \emph{rooted planar $3$-tree} is any plane graph obtained from a 3-cycle $C$ by repeatedly selecting an internal triangular face,
adding a new vertex inside the face, and connecting the new vertex to each vertex of the face containing it; see Figure~\ref{fig-stacks} for an example.
Note that $C$ bounds the outer face of the resulting graph.  For a cycle $K$ in a plane graph $G$, let $G_K$ denote the subgraph of $G$ drawn in the closed disk bounded
by $K$.  Let us note the following important property of rooted planar $3$-trees.
\begin{observation}\label{obs-treelike}
If $G$ is a rooted planar $3$-tree with the outer face bounded by the triangle $C$ and $G\neq C$, then
there exists a unique vertex $v\in V(G)\setminus V(C)$ adjacent to all vertices of $C$, and for each triangle $K$ bounding
an internal face of the subgraph $G[V(C)\cup \{v\}]$, the graph $G_K$ is a rooted planar $3$-tree.
\end{observation}

In the situation of Observation~\ref{obs-treelike}, we say that $v$ is the \emph{tip} of $G$ and the graphs $G_K$ are the \emph{children} of $G$.
Rooted planar $3$-trees are problematic from the precoloring extension perspective: For a non-negative integer $k$, let $G_k$
denote the complete rooted planar $3$-tree of depth $k$, i.e., $G_0$ is a $3$-cycle and for $k\ge 1$, $G_k$ is the rooted planar $3$-tree
whose children are all copies of $G_{k-1}$.  Let $C=v_1v_2v_3$ be the triangle bounding the outer face of $G_k$ and let $\psi$ be the precoloring of $C$
such that $\psi(v_i)=i$ for $i\in\{1,2,3\}$.  Then no matter which of the colors $1$, $2$, or $3$ we use to color the tip
of $G_k$, there will be a child of $G_k$ such that all three colors appear on the triangle bounding its outer face.  Using this observation,
an easy inductive argument shows that in any $3$-coloring of $G_k$ extending $\psi$, the sum of the sizes of the clusters containing $v_1$,
$v_2$, and $v_3$ is at least $k+3$, and thus the coloring cannot have clustering smaller than $k/3+1$.  However, note that $G_k$ contains $P''_{k+1}$
as a subgraph, and fortunately $P''_t$-free rooted planar $3$-trees behave much nicer with respect to clustered coloring, as shown in the following lemma.

\begin{lemma}\label{lemma-stacks}
Let $t$ be a positive integer, let $G$ be a rooted planar $3$-tree and let $L$ be a $3$-list-assignment for $G$, and let $r$
be a vertex incident with the outer face of $G$. If $G$ is $P''_t$-free, then every $L$-coloring $\psi$ of the cycle $C$
bounding the outer face of $G$ extends to an $L$-coloring $\varphi$ of $G$ with clustering at most $3t$
such that either
\begin{itemize}
\item $\varphi$ isolates $V(C)$, or
\item $\psi$ uses three distinct colors on $C$, the tip of $G$ has list $\{\psi(u):u\in V(C)\}$,
$\varphi$ isolates $V(C)\setminus\{r\}$, and the cluster of $\varphi$ containing $r$ has size less than $t$.
\end{itemize}
\end{lemma}
\begin{proof}
We prove the claim by induction on the number of vertices of $G$.  Let $C=rxy$.
The claim is clear if $G=C$, and thus assume that $V(G)\neq V(C)$.  Let $v$ be the tip of $G$.

If there exists a color $c\in L(v)\setminus\{\psi(r),\psi(x),\psi(y)\}$, then we color $v$ by $c$ and
extend this coloring to the children $G_{rxv}$, $G_{ryv}$, $G_{xyv}$ by the induction hypothesis,
with $v$ playing the role of the vertex $r$.  Observe that the resulting coloring $\varphi$ isolates $V(C)$ and that
the cluster containing $v$ has size at most $3(t-1)-2<3t$, and thus $\varphi$ has clustering at most $3t$.

\begin{figure}
\begin{center}
\includegraphics{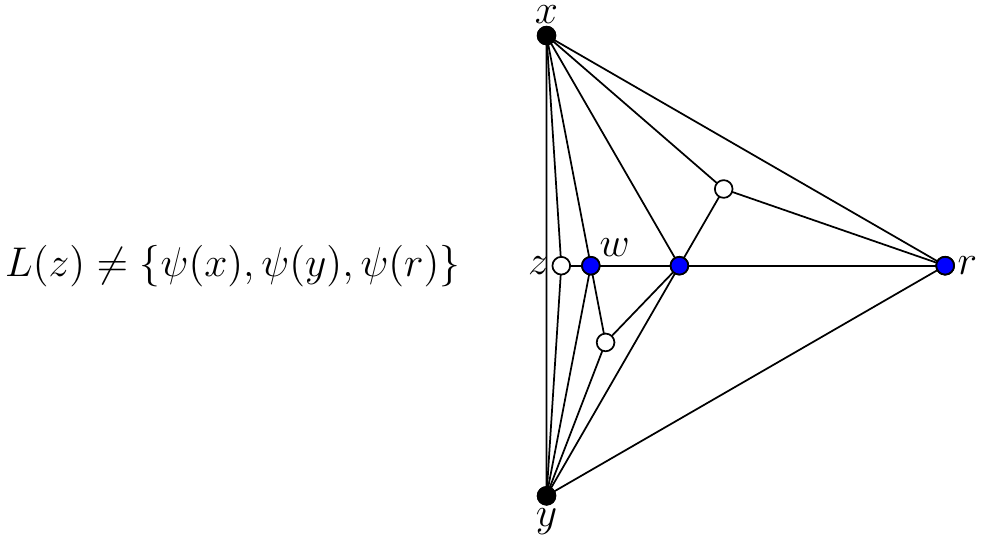}
\end{center}
\caption{The path $P$ from the proof of Lemma~\ref{lemma-stacks}, drawn in blue.}\label{fig-longest}
\end{figure}

Hence, we can assume that $L(v)=\{\psi(r),\psi(x),\psi(y)\}$, and in particular $\psi$ uses three distinct colors on $C$.
Let $P$ be the longest path in $G$ starting in $r$ such that all vertices of $P$ are adjacent to both $x$ and $y$
and all vertices in $V(P)\setminus\{r\}$ have list $\{\psi(r),\psi(x),\psi(y)\}$, see Figure~\ref{fig-longest} for an illustration.
Note that $|V(P)|\ge 2$, and since $G$ is $P''_t$-free, $P$ has at most $t-1$ vertices.  Let $\varphi(u)=\psi(r)$ for 
$u\in V(P)$.  Consider any triangle $K$ bounding an internal face of $G[V(C)\cup V(P)]$.  If $K$ contains an edge of $P$,
then two of its vertices have the same color $\psi(r)$.  Otherwise $K=wxy$, where $w$ is the last vertex of $P$,
and by the maximality of $P$, if $G_K\neq K$, then the tip $z$ of $G_K$ has list different from $\{\varphi(w),\psi(x),\psi(y)\}$.
Therefore, in either case the induction hypothesis shows that the coloring of $K$ extends to an $L$-coloring of $G_K$ of clustering at most $3t$
which isolates $V(K)$.  Combining these colorings gives an $L$-coloring of $G$ of clustering at most $3t$ such that the only clusters intersecting $V(C)$
are $\{x\}$, $\{y\}$, and $V(P)$.  Since $|V(P)|<t$, the conclusion of the lemma holds.
\end{proof}

Let us now consider a graph $G'$ obtained from a graph $G$ drawn on a surface by filling in triangular faces of $G$ by rooted planar $3$-trees.
We want to use Lemma~\ref{lemma-stacks} in order to extend a coloring of $G$ to a coloring of $G'$ without increasing the clustering substantially.
To do so, we need to ensure that a single vertex of $G$ is not chosen to play the role of $r$ in too many incident triangles.
Let $Z$ be a set of faces of $G$.  A function $\pi:Z\to V(G)$ is a \emph{pointer system}
for $Z$ if for each $f\in Z$, the vertex $\pi(f)$ is incident with $f$.
If each vertex is assigned to at most $\theta$ faces, we say that $\pi$ has
\emph{thickness} at most $\theta$.
\begin{observation}\label{obs-pointer}
Let $G$ be a graph drawn on a surface of Euler genus $g$, let $Z$ be a set of faces of $G$, and let $X$ be a set of vertices
of $G$ such that each face in $Z$ is incident with a vertex not belonging to $X$.  Then there exists a pointer system $\pi$
for $Z$ of thickness at most $12(g+1)+2|X|$ such that $\pi(Z)\cap X=\emptyset$.
\end{observation}
\begin{proof}
By generalized Euler's formula, every graph drawn on a surface of Euler genus $g$ has average degree less than $6(g+1)$
(actually, the tight bound is given by Heawood's formula, but for our purposes even this very loose bound suffices).
Consequently, there exists an ordering $\prec$ of $V(G)$ such that each vertex has less than $6(g+1)$ neighbors that
appear after it in the ordering.  For each face $f\in Z$, let $\pi(f)$ be the first vertex incident with $f$ in the ordering $\prec$
that does not belong to $X$.  Observe that for each non-isolated vertex $v\in V(G)$, any face $f\in Z$ such that $\pi(f)=v$ is incident
either with an edge from $v$ to $X$, or from $v$ to a vertex $u$ such that $v\prec u$.  There are at most $6(g+1)+|X|$ such edges
and each of them is incident with at most two faces, giving the desired bound on the thickness of $\pi$.
\end{proof}

Let $G$ be a graph drawn on a surface $\Sigma$ and let $X$ be a set of vertices of $G$.  Moreover, if $\Sigma$ is the
sphere, assume that $|X|\ge 3$.
An $X$-external triangle $K$ in $G$ is \emph{$X$-empty} if $K$ is contractible and an open disk $\Lambda$
bounded by $K$ does not contain any vertex of $X$.  Such a disk is clearly unique when $\Sigma$ is not the sphere.
If $\Sigma$ is the sphere, it is also unique due to the additional assumption $|X|\ge 3$ (since $K$ is $X$-external,
we have $X\not\subseteq V(K)$).
We define $G_{K,X}$ to be the subgraph of $G$ drawn in the closure of $\Lambda$.  We say that $K$ is
\emph{$X$-pyramidal} if $K$ is $X$-empty and $V(K)$ is not $3$-solitary in $G_{K,X}$,
i.e., there exists a vertex in $G_{K,X}$ adjacent to all three vertices of $K$.
Combining the previous results, we obtain the following claim.
\begin{corollary}\label{cor-onlystacks}
For every surface $\Sigma$ and integer $t$, there exists a function $\gamma_{\ref{cor-onlystacks}}:\mathbb{N}\to\mathbb{N}$
such that the following claim holds.
Let $G$ be a graph drawn on $\Sigma$ and let $X$ be a set of vertices of $G$ (where $|X|\ge 3$ if $\Sigma$ is the sphere)
such that every $X$-external non-facial triangle in $G$ is $X$-pyramidal.  Let $L$ be a $3$-list-assignment for $G$.  If $G$ is $P''_t$-free,
then every $L$-coloring $\psi$ of $X$ extends to an $L$-coloring $\varphi$ of $G$ with clustering at most $\gamma_{\ref{cor-onlystacks}}(|X|)$
such that if $X$ is $3$-solitary, then $\varphi$ isolates $X$.
\end{corollary}
\begin{proof}
Let $g$ be the Euler genus of $\Sigma$, and let $\gamma_{\ref{cor-nocut}}:\mathbb{N}\to \mathbb{N}$ be the function from
Corollary~\ref{cor-nocut}.  Let $\gamma_{\ref{cor-onlystacks}}(p)=(12(g+1)+2p)t\gamma_{\ref{cor-nocut}}(p)$.

Observe that since each $X$-external non-facial triangle in $G$ is $X$-pyramidal,
$G_{K,X}$ is a rooted planar $3$-tree for every $X$-external non-facial triangle $K$.
Let $K_1$, \ldots, $K_m$ be $X$-external non-facial triangles in $G$ such that the open disks disjoint from $X$
bounded by them are inclusionwise-maximal, and observe that these open disks are disjoint.
Let $G_0$ be the graph obtained from $G$ by deleting vertices and edges
drawn in these disks.  Then every $X$-external triangle in $G_0$ is facial, and by Corollary~\ref{cor-nocut},
$\psi$ extends to an $L$-coloring $\varphi_0$ of $G_0$ with clustering at most $\gamma_{\ref{cor-nocut}}(|X|)$ such that if $X$ is $3$-solitary,
then $\varphi_0$ isolates $X$.  Let $Z$ be the set of faces of $G_0$ formed by the open disks disjoint from $X$ bounded by $K_1$, \ldots, $K_m$.
Since the triangles bounding these disks are $X$-external, by Observation~\ref{obs-pointer}, there exists
a pointer system $\pi$ for $Z$ of thickness at most $12(g+1)+2|X|$ such that $\pi(Z)\cap X=\emptyset$.
We extend $\varphi_0$ to each of the rooted planar $3$-trees $G_{K_1}$, \ldots, $G_{K_m}$ using Lemma~\ref{lemma-stacks},
where the vertex $r$ is chosen according to the pointer system $\pi$ (and in particular does not belong to $X$),
obtaining an $L$-coloring $\varphi$ of $G$ that extends $\psi$.
Since $\pi$ has thickness at most $12(g+1)+2|X|$, a cluster of $\varphi_0$ of size $k$ grows at most
to size $(12(g+1)+2|X|)tk$ in $\varphi$.  Therefore, $\varphi$ has clustering at most
$(12(g+1)+2|X|)t\gamma_{\ref{cor-nocut}}(|X|)=\gamma_{\ref{cor-onlystacks}}(|X|)$,
and moreover if $X$ is $3$-solitary, then $\varphi$ isolates $X$.
\end{proof}

\section{Contractible triangles}\label{sec-extend}

Next, let us deal with the planar case of the general problem, in a technical setting designed to deal with
separating triangles by a precoloring extension argument.  We say a non-empty set $X$ in a plane graph $G$
is \emph{cut off from the outer face} if there exists a triangle $K$ in $G$ that does not bound
the outer face and $X\subseteq V(G_K)\setminus V(K)$.  Let $q_G(X)=2|X|+1$ if $X$ is cut off from the outer face
and $q_G(X)=2|X|$ otherwise.  Let us also define $q_G(\emptyset)=0$.

\begin{lemma}\label{lemma-ext}
For every integer $t$, there exists a function $\gamma_{\ref{lemma-ext}}:\mathbb{N}\to\mathbb{N}$
such that the following claim holds.  Let $G$ be a plane graph with the outer face bounded by a triangle $C$,
let $X$ be a set of vertices of $G$ containing $V(C)$ and let $L$ be a $3$-list-assignment for $G$.
If $G$ is $P''_t$-free, then every $L$-coloring $\psi$ of $X$ extends to an $L$-coloring $\varphi$ of $G$ with clustering
at most $\gamma_{\ref{lemma-ext}}(q_G(X\setminus V(C)))$
such that if $X=V(C)$ and $X$ is $3$-solitary, then $\varphi$ isolates $X$.
\end{lemma}
\begin{proof}
Let $\gamma_{\ref{cor-onlystacks}}:\mathbb{N}\to\mathbb{N}$ be the function from Corollary~\ref{cor-onlystacks} with $\Sigma$ being the sphere;
without loss of generality, we can assume that this function is non-decreasing.
Let $\gamma_{\ref{lemma-ext}}(m)=2^m\gamma_{\ref{cor-onlystacks}}(m+3)$.

We prove the claim by induction on the number of vertices of $G$.  Suppose first that there exists an $X$-external 
triangle in $G$ that is not $X$-empty (this implies that $X\neq V(C)$), and let $K$ be such a cycle with $G_K$ minimal.
Let $G'=G-(V(G_K)\setminus V(K))$ and $X'=X\cap V(G')$.
By the induction hypothesis, the restriction of $\psi$ to $X'$ extends to
an $L$-coloring $\varphi'$ of $G'$ with clustering at most $\gamma_{\ref{lemma-ext}}(q_{G'}(X'\setminus V(C)))$.
Let $X''=V(K)\cup (X\cap V(G_K))$ and let $\psi''$ be the $L$-coloring of $X''$ such that
$\psi''(v)=\varphi'(v)$ for $v\in V(K)$ and $\psi''(v)=\psi(v)$ for $v\in X''\setminus V(K)$.
By the induction hypothesis, $\psi''$ extends to an $L$-coloring $\varphi''$ of $G_K$
with clustering at most $\gamma_{\ref{lemma-ext}}(X''\setminus V(K))$.  Let $\varphi$ be the combination of $\varphi'$ and $\varphi''$.
Only the clusters intersecting $K$ can be merged, and since $K$ is a clique, only one cluster of $\varphi'$
and one cluster of $\varphi''$ is contained in the merged cluster. 
Note that $|X'\setminus V(C)|<|X\setminus V(C)|$ since $K$ is not $X$-empty, and thus $q_{G'}(X'\setminus V(C))<q_{G}(X\setminus V(C))$.
Moreover, either $|X''\setminus V(K)|<|X\setminus V(C)|$, or $X''\setminus V(K)=X\setminus V(C)$ and $K$
shows that $X\setminus V(C)$ is cut off from the outer face of $G$.  In the latter case, the minimality of $G_K$ implies
that $X''\setminus V(K)$ is not cut off from the outer face of $G_K$.  In either case, we conclude that $q_{G_K}(X''\setminus V(K))<q_G(X\setminus V(C))$.
Therefore, the clustering of $\varphi$ is
at most $\gamma_{\ref{lemma-ext}}(q_{G'}(X'\setminus V(C)))+\gamma_{\ref{lemma-ext}}(q_{G_K}(X''\setminus V(K)))\le \gamma_{\ref{lemma-ext}}(q_G(X\setminus V(C)))$.

Suppose now that $G$ contains an $X$-external non-facial $X$-empty triangle $K$ that is not $X$-pyramidal,
i.e., $V(K)$ is $3$-solitary in $G_K$.  Let $G'$, $X'=X$, $X''=V(K)$, $\varphi'$, $\varphi''$, and $\varphi$ be as in the previous paragraph.
The induction hypothesis now ensures that $\varphi''$ isolates $X''$, and thus the clustering of $\varphi$ is
at most $\max(\gamma_{\ref{lemma-ext}}(q_{G'}(X'\setminus V(C))),\gamma_{\ref{lemma-ext}}(q_{G_K}(\emptyset)))=\gamma_{\ref{lemma-ext}}(q_G(X\setminus V(C)))$.
Moreover, if $X=V(C)$ and $X$ is $3$-solitary in $G$, it is also $3$-solitary in $G'$, and thus
$\varphi'$ isolates $X$. In that case, it follows that $\varphi$ isolates $X$ as well.

Therefore, we can assume that every $X$-external non-facial triangle in $G$ is $X$-pyramidal.  By Corollary~\ref{cor-onlystacks},
$\psi$ extends to an $L$-coloring $\varphi$ of $G$ with clustering at most $\gamma_{\ref{cor-onlystacks}}(|X|)\le \gamma_{\ref{lemma-ext}}(q_G(X\setminus V(C)))$
such that if $X$ is $3$-solitary, then $\varphi$ isolates $X$.
\end{proof}

Using this lemma, it is easy to deal with non-facial contractible triangles in graphs on surfaces.

\begin{lemma}\label{lemma-extsurf}
For every surface $\Sigma$ and integer $t$, there exists a function $\gamma_{\ref{lemma-extsurf}}:\mathbb{N}\to\mathbb{N}$
such that the following claim holds.  Let $G$ be a graph drawn on $\Sigma$ without non-contractible triangles,
let $X$ be a set of vertices of $G$ and let $L$ be a $3$-list-assignment for $G$.
If $G$ is $P''_t$-free, then every $L$-coloring $\psi$ of $X$ extends to an $L$-coloring of $G$ with clustering
at most $\gamma_{\ref{lemma-extsurf}}(|X|)$.
\end{lemma}
\begin{proof}
Let $\gamma_{\ref{lemma-ext}}:\mathbb{N}\to\mathbb{N}$ be the function from Lemma~\ref{lemma-ext}, and
$\gamma_{\ref{cor-onlystacks}}:\mathbb{N}\to\mathbb{N}$ the function from Corollary~\ref{cor-onlystacks}.
We can assume that both of these functions are non-decreasing.
Let $\gamma_{\ref{lemma-extsurf}}(p)=(p+1)\gamma_{\ref{lemma-ext}}(2p+1)+\gamma_{\ref{cor-onlystacks}}(p)$.

If $\Sigma$ is the sphere, then the claim follows from Lemma~\ref{lemma-ext} applied to the disjoint union of $G$ with a triangle $C$
bounding the outer face.  Hence, suppose that $\Sigma$ is not the sphere.
We prove the claim by induction on $|V(G)|$.  If every $X$-external non-facial triangle in $G$ is $X$-pyramidal, then the claim
follows from Corollary~\ref{cor-onlystacks}.  Hence, we can assume that $G$ contains an $X$-external non-facial triangle $K$
which is not $X$-pyramidal, i.e., denoting by $\Lambda$ the open disk bounded by $K$ and by $G''$ the subgraph of $G$ drawn in the closure of $K$,
either $\Lambda$ contains a vertex of $X$ or $V(K)$ is $3$-solitary in $G''$.  Let $G'$ be the graph obtained from $G$ by
deleting the vertices and edges drawn in $\Lambda$ and let $X'=X\cap V(G')$.  By the induction hypothesis,
the restriction of $\psi$ to $X'$ extends to an $L$-coloring $\varphi'$ of $G'$ with clustering at most $\gamma_{\ref{lemma-extsurf}}(|X'|)$.
Let $X''=V(K)\cup (X\cap V(G''))$ and let $\psi''$ be the $L$-coloring of $X''$ matching $\varphi'$ on $K$ and
$\psi$ on the rest of the vertices.  By Lemma~\ref{lemma-ext}, $\psi''$ extends to an $L$-coloring $\varphi''$ of $G''$ with
clustering at most $\gamma_{\ref{lemma-ext}}(q_{G''}(X''\setminus V(K)))\le \gamma_{\ref{lemma-ext}}(2|X|+1)$,
where $\varphi''$ isolates $V(K)$ if $X''=V(K)$.  Let $\varphi$ be the union of $\varphi'$ and $\varphi''$.
If $K$ is not $X$-empty, then $|X'|<|X|$ and $\varphi$ has clustering at most
$\gamma_{\ref{lemma-extsurf}}(|X'|)+\gamma_{\ref{lemma-ext}}(2|X|+1)\le \gamma_{\ref{lemma-extsurf}}(|X|-1)+\gamma_{\ref{lemma-ext}}(2|X|+1)\le \gamma_{\ref{lemma-extsurf}}(|X|)$.
Otherwise $\varphi''$ isolates $V(K)$ and $\varphi$ has clustering at most
$\max(\gamma_{\ref{lemma-extsurf}}(|X'|),\gamma_{\ref{lemma-ext}}(2|X|+1))=\gamma_{\ref{lemma-extsurf}}(|X|)$.
\end{proof}

\section{The general case}\label{sec-general}

The proof of our main result is finished by using a precoloring extension argument again to deal with non-contractible triangles.

\begin{proof}[Proof of Theorem~\ref{thm-main}]
Let $\gamma_{\ref{lemma-extsurf}}:\mathbb{N}\to\mathbb{N}$ be the function from Lemma~\ref{lemma-extsurf}.
We prove the theorem by induction on the Euler genus $g$ of $\Sigma$; hence, suppose that the claim holds for all surfaces
of smaller Euler genus.  Since there are only finitely many such surfaces, it follows that there exists
a function $\gamma':\mathbb{N}\to\mathbb{N}$ such that every $P''_t$-free graph drawn on a surface of Euler genus less than $g$ is
$p$-precoloring-$3$-choosable with clustering at most $\gamma'(p)$ for every integer $p$.
Without loss of generality, we can assume that both functions $\gamma_{\ref{lemma-extsurf}}$ and $\gamma'$ are non-decreasing.
Let $\gamma(p)=\max(\gamma_{\ref{lemma-extsurf}}(p),2\gamma'(p+6))$.

Let $G$ be a $P''_t$-free graph drawn on $\Sigma$, let $X$ be a set of vertices of $G$, let $L$ be a $3$-list-assignment for $G$
and let $\psi$ be an $L$-coloring of $X$.   If $G$ has no non-contractible triangle, then by Lemma~\ref{lemma-extsurf},
$\psi$ extends to an $L$-coloring of $G$ of clustering at most $\gamma_{\ref{lemma-extsurf}}(|X|)\le \gamma(|X|)$.

Hence, suppose that $K$ is a non-contractible triangle in $G$.  Let $X'=X\cup V(K)$ and let us extend $\psi$ to an $L$-coloring
of $X'$ arbitrarily.  Cutting $\Sigma$ along $K$ and patching the resulting holes gives us a graph $G_0$
drawn on a surface (or two surfaces, in case $K$ separates $\Sigma$ into two parts) of Euler genus less than $g$.
Let $X_0$ be the set of vertices of $G_0$ corresponding to $X'$ (where each vertex of $K$ corresponds to two vertices),
with coloring $\psi_0$ matching $\psi$ on the corresponding vertices.  By the induction hypothesis, $\psi_0$
extends to an $L$-coloring of $G_0$ with clustering at most $\gamma'(|X_0|)\le \gamma'(|X|+6)$.
Gluing the two copies of $K$ in $G_0$ back together gives an $L$-coloring of $G$ that extends $\psi$ with clustering
at most $2\gamma'(|X|+6)\le \gamma(|X|)$.
\end{proof}

\section*{Acknowledgements}

I would like to thank David Wood for organizing the MATRIX-IBS workshop ``Structural Graph Theory Downunder III'',
where I learned about the problem, and to him and Chun-Hung Liu for fruitful discussions.

\bibliographystyle{plain}
\bibliography{../data.bib}

\begin{thebibliography}{10}

\bibitem{cowen1986defective}
Lenore~J Cowen, Robert~H Cowen, and Douglas~R Woodall.
\newblock Defective colorings of graphs in surfaces: partitions into subgraphs
  of bounded valency.
\newblock {\em Journal of Graph Theory}, 10:187--195, 1986.

\bibitem{delcourt2021reducing}
Michelle Delcourt and Luke Postle.
\newblock Reducing linear {H}adwiger's conjecture to coloring small graphs.
\newblock {\em arXiv}, 2108.01633, 2021.

\bibitem{clushadalt}
Vida Dujmovi{\'c}, Louis Esperet, Pat Morin, and David~R. Wood.
\newblock Proof of the clustered {H}adwiger conjecture.
\newblock {\em arXiv}, 2023.

\bibitem{islands}
Zden\v{e}k Dvo\v{r}\'ak and Sergey Norin.
\newblock Islands in minor-closed classes. {I}. {B}ounded treewidth and
  separators.
\newblock {\em arXiv}, 1710.02727, 2017.

\bibitem{weak}
Zden\v{e}k Dvo\v{r}\'ak and Sergey Norin.
\newblock Weak diameter coloring of graphs on surfaces.
\newblock {\em arXiv}, 2111.07147, 2021.

\bibitem{edwards2014relative}
Katherine Edwards, Dong~Yeap Kang, Jaehoon Kim, Sang-il Oum, and Paul Seymour.
\newblock A relative of {H}adwiger's conjecture.
\newblock {\em {SIAM} J. Discrete Math.}, 29:2385--2388, 2015.

\bibitem{espjor}
L.~Esperet and G.~Joret.
\newblock Colouring planar graphs with three colours and no large monochromatic
  components.
\newblock {\em Combinatorics, Probability and Computing}, 23:551--570, 2014.

\bibitem{kawarabayashi2007relaxed}
Ken-ichi Kawarabayashi and Bojan Mohar.
\newblock A relaxed {H}adwiger's conjecture for list colorings.
\newblock {\em Journal of Combinatorial Theory, Series B}, 97:647--651, 2007.

\bibitem{liusmall}
Chun{-}Hung Liu and Sang{-}il Oum.
\newblock Partitioning {$H$}-minor free graphs into three subgraphs with no
  large components.
\newblock {\em Electronic Notes in Discrete Mathematics}, 49:133--138, 2015.

\bibitem{liu2019clustered}
Chun-Hung Liu and David~R Wood.
\newblock Clustered graph coloring and layered treewidth.
\newblock {\em arXiv}, 1905.08969, 2019.

\bibitem{robertsonseymourthomas}
N.~Robertson, P.~D. Seymour, and R.~Thomas.
\newblock {Hadwiger's conjecture for $K_6$-free graphs}.
\newblock {\em Combinatorica}, 13:279--361, 1993.

\bibitem{thomassen1994}
C.~Thomassen.
\newblock Every planar graph is 5-choosable.
\newblock {\em J. Combin. Theory, Ser.~B}, 62:180--181, 1994.

\bibitem{van2018improper}
Jan van~den Heuvel and David~R. Wood.
\newblock Improper colourings inspired by {H}adwiger's conjecture.
\newblock {\em Journal of the London Mathematical Society}, 98:129--148, 2018.

\bibitem{voigt1993}
M.~Voigt.
\newblock List colourings of planar graphs.
\newblock {\em Discrete Math.}, 120:215--219, 1993.

\bibitem{wood2010contractibility}
David~R Wood.
\newblock Contractibility and the {H}adwiger conjecture.
\newblock {\em European Journal of Combinatorics}, 31:2102--2109, 2010.

\bibitem{wood2018defective}
David~R Wood.
\newblock Defective and clustered graph colouring.
\newblock {\em The Electronic Journal of Combinatorics}, 1000:23--13, 2018.

\end{thebibliography}

\end{document}